\documentclass[a4paper,11pt]{article}

\title{Generation of Iterated Wreath Products Constructed from
  Almost Simple Groups}
\author{Jiaping Lu \\[10pt]
  Mathematical Institute, University of St Andrews,\\
  North Haugh, St Andrews, Fife, KY16 9SS\\[10pt]
  \begin{tabular}{c}\texttt{jl337@st-andrews.ac.uk}\\
    \texttt{jiapingljp@hotmail.com}
  \end{tabular}}

\usepackage{amsmath,amssymb}
\usepackage[margin=2.9cm]{geometry}
\usepackage[amsmath,thmmarks]{ntheorem}
\usepackage{enumitem}
\usepackage{tikz}
\usepackage{mathtools}
\usepackage{tikz}
\usepackage{booktabs}

\theorembodyfont{\slshape}
\newtheorem{lem}{Lemma}[section]

\newtheorem{thm}[lem]{Theorem}
\theorembodyfont{\normalfont}

\theorembodyfont{\slshape}
\theoremnumbering{Alph}

\makeatletter
\theoremstyle{nonumberplain}
\theorembodyfont{\normalfont}
\theoremheaderfont{\normalfont\scshape}
\theoremsymbol{~\ensuremath\square}
\newtheoremstyle{proofstyle}%
  {\item[\theorem@headerfont\hskip\labelsep ##1\theorem@separator]}%
  {\item[\theorem@headerfont\hskip\labelsep ##3\theorem@separator]}
\theoremstyle{proofstyle}
\newtheorem{prf}{Proof:}
\makeatother

\renewcommand{\leq}{\leqslant}
\renewcommand{\geq}{\geqslant}
\newcommand{\nbd}{\nobreakdash-}

\newcommand{\trivsubgp}{\mathbf{1}}
\newcommand{\Cent}[2]{\mathrm{C}_{#1}(#2)}
\newcommand{\Cohom}[2][1]{\mathrm{H}^{#1}(#2)}
\newcommand{\Der}[1]{\mathrm{Der}(#1)}
\newcommand{\End}[2][G]{\operatorname{End}_{#1}(#2)}
\newcommand{\Field}[1]{\mathbb{F}_{#1}}
\newcommand{\Frat}[1]{\Phi(#1)}
\newcommand{\Inn}[1]{\mathrm{Inn}(#1)}

\newcommand{\order}[1]{\mathopen{|}#1\mathclose{|}}
\newcommand{\Out}{\operatorname{Out}}
\newcommand{\Outdiag}{\operatorname{Outdiag}}
\newcommand{\pr}[1]{\operatorname{pr}(#1)}

\newcommand{\set}[2]{\{\,#1\mid#2\,\}}
\newcommand{\Zint}{\mathbb{Z}}

\renewcommand{\geq}{\geqslant}
\renewcommand{\leq}{\leqslant}

\renewcommand{\wr}{\operatorname{wr}}

\setlist[enumerate,1]{label={\normalfont(\roman*)}}

\usepackage{hyperref}
\hypersetup{colorlinks=true, linkcolor={red!50!black},
  citecolor={green!50!black}}

\begin{document}

\maketitle

\begin{abstract}
  Let $G_{1}$,~$G_{2}$, \dots\ be a sequence of almost simple groups  and construct a sequence~$(W_{i})$ of wreath products via $W_{1} =
  G_{1}$ and, for each $i >1$, $W_{i+1} = G_{i+1} \wr W_{i}$ via
  the regular action of each $G_i$.  We determine the minimum
  number~$d(W_{i})$ of generators required for each wreath product in
  this sequence.
\end{abstract}


\section{Introduction}
A classic theme of group theory research is the generation property of groups. Elements of a group can be described using its generators. There has been extensive research on generators of specific groups. For example, it is taught in undergraduate group theory courses that symmetric groups can be generated with two elements. Using the Classification of Finite Simple Groups, it is proved that every non-abelian finite simple group is $2$\nbd generated.

One of the interesting generation problems is the growth rate of the minimum number of generators of the direct product $G^n$ where $G$ is a finite group. Wiegold shows in \cite{Wiegold} that the number of required generators grows linearly when $G$ is an imperfect group and it grows logarithmically when $G$ is perfect. Apart from direct products, another way to construct groups is via wreath products. In this paper, we will investigate the minimum number of generators of iterated wreath products.

Bhattacharjee \cite{Bhatt} investigates the generation property of iterated wreath products of alternating groups constructed via the natural action. She proves that the probability that such products can be generated by two random elements is positive. This shows that iterated wreath products of alternating groups are $2$\nbd generated. Quick \cite{MRQ} generalises this result and shows that the same result still holds for iterated wreath products of non-abelian simple groups constructed from any faithful transitive action of the factors. Bondarenko \cite{Bond} considers  iterated wreath products of transitive permutation groups with uniformly bounded number of generators. He establishes that the number of generators of the product involving $G_1, \dots, G_n$ is bounded if and only if the number of generators of the abelian direct product $G_1/G_1'\times G_2/G_2'\times\dots\times G_n/G_n'$ is bounded as well.

Recently, the author of this paper \cite{LuQuick} considered the generators required for the iterated wreath products of symmetric groups, alternating groups or cyclic groups. In particular, it is established that the minimum number of generators of iterated wreath products of symmetric groups and alternating groups of degree at least $5$ is the maximum of $2$ and the number of symmetric group wreath factors involved. In this article, we will generalise this result by allowing wreath factors to be any almost simple group. When determining the minimum number of generators of iterated wreath products in \cite{LuQuick}, calculations based on actions of wreath factors are necessary. There are many choices for a transitive and faithful action of an almost simple group. However, in order to be uniform and consistent among all almost simple groups, we will consider the regular action of each wreath factor. It is worth noting that Woryna \cite{Woryna} investigates the generation of iterated wreath products of abelian groups, where he assumed each abelian group factor acts faithfully and transitively and hence the action is also regular in his case.

This paper will establish the precise number of generators required for an iterated wreath product constructed from the regular action of almost simple groups.  If we take the factors involved to be simple groups, then we recover the observation that iterated wreath products constructed via the regular action of non-abelian simple groups are $2$\nbd generated. Our main result depends on the minimum number of generator of the wreath product $A\wr G$ of an abelian group $A$ constructed from the regular action of an almost simple group $G$. In \cite{Lu97}, Lucchini establishes $d(A\wr G)$ when $G$ is simple, and we will use this result as part of our argument.

To state the theorem, we recall some basic notation. For a group $G$, we write $d(G)$ for the minimum number of generators for $G$. For a finite abelian group $A$ and a prime number $p\mid\order{A}$, we write $d_p(A)$ for the minimum number of generators for the Sylow $p$\nbd subgroup of $A$. Because of the Fundamental Theorem of Finite Abelian Groups, $d(A)=\max_{p\mid\order{A}}d_p(A)$. We will establish the following theorem.
\begin{thm}\label{thm:main}
Let $G_{1}$,~$G_{2}$, \dots,~$G_{k}$ be a sequence of almost simple groups. Let $S$ be the socle of $G_1.$ Let $W = G_{k} \wr G_{k-1} \wr \dots \wr G_{1}$ 
be the iterated wreath product constructed via the regular action of each factor. Set
\[
A = (G_{k}/G_{k}') \times (G_{k-1}/G_{k-1}') \times \dots \times (G_{2}/G_{2}').
\]
Then
\[
d(W)=\max_p(d(A\times G_1), d(A)+1, d_p(A)+2),
\]
where $p$ ranges over the set of prime numbers dividing both $\order{A}$ and $\order{S}$.
\end{thm}

We finish this introduction by giving the structure of this paper. In Section~\ref{sec:pre}, we recall basic definitions that we will use. We will also introduce the results that our conclusions depend on, particularly the ones from \cite{Lu97}. In Section~\ref{sec:AwrG}, using the results in~\cite{Lu97} and also extending them to cover a more general case, we will determine the minimum number of generators of the wreath product of an abelian group by an almost simple group constructed via the regular action. We will then provide the proof of the main theorem in the final section. We will prove it by showing that $d(W)=d(A\wr G_1)$ and determining $d(A\wr G_1)$, where $W, A, G_1$ are as defined in Theorem~\ref{thm:main} (see Lemmas~\ref{lem:keylemmaregular-abelian} and \ref{lem:regularinduction} respectively).

\paragraph{Acknowledgements:} The author is funded by the China
Scholarship Council. We thank the referee for their suggestions.

\section{Preliminaries}\label{sec:pre}
In this section, we introduce the notation that we will use and summarise the
facts from previous work that we rely upon.  We start with the
definition of wreath products and its associated notation.

Let $G$~and~$H$ be permutation groups on the finite sets $X$~and~$Y$,
respectively.  Let $B = \prod_{y \in Y} G_{y}$ be the direct product
of copies of~$G$ indexed by the set~$Y$.  We shall write $b =
(g_{y}) = (g_{y})_{y \in Y}$ for an element of~$B$ expressed as a
sequence of elements from~$G$ indexed by~$Y$.  We define an action
of~$H$ on~$B$ by permuting the entries of its elements in the same way
that $H$~acts on~$Y$: $(g_{y})^{h} = (g_{yh^{-1}})$ for $b = (g_{y}) \in B$ and $h \in H$.  The (permutational) wreath
product $W = G \wr_{Y} H$ of~$G$ by~$H$ is the semidirect product $W =
B \rtimes H$ constructed via this action.  In this article, we shall
construct wreath products from the regular action of $H$.  Consequently, we shall omit the
subscript on the wreath product and simply write $W = G \wr H$ for the
wreath product of~$G$ by~$H$ constructed via the regular action of the
permutation group~$H$.

The wreath product $W = G \wr H$ has a natural action on the Cartesian
product~$X \times Y$ given by $(x,y)^{g} = (x^{g_{y}}, y^{h})$ for $(x,y) \in X \times Y$ and $g = (g_{y})h \in W$.  As a
consequence, if $G_{1}$,~$G_{2}$, \dots,~$G_{k}$ is a sequence of
permutation groups acting regularly on the sets $X_{1}$,~$X_{2}$,
\dots,~$X_{k}$, respectively, then we may iterate the wreath product
construction using this induced action.  Thus we recursively
define $W_{1} = G_{1}$ acting upon the set~$X_{1}$ and then, for $2\leq i\leq k$, having
defined~$W_{i-1}$ with an action on $X_{i-1} \times \dots \times
X_{1}$, we define $W_{i} = G_{i} \wr W_{i-1}$ with its action on
$X_{i} \times X_{i-1} \times \dots \times X_{1}$.  

Let $d(G)$ be the minimum number of generators of a finite group $G$. We have the following important result.
\begin{thm}[Lucchini, Menegazzo~{\cite{LuMe}}]\label{thm:LuMe}
If $G$ is a non-cyclic finite group with a unique minimal normal subgroup $N$, then $d(G)=\max(2, d(G/N))$.
\end{thm}

Our work relies upon the results in Lucchini's work~\cite{Lu97}. In the following, we will recall the notation required and introduce useful consequences.

A chief factor of a finite group~$G$ is the quotient~$H/K$ of normal subgroups $H$~and~$K$ with $K < H$ such that there is no normal subgroup~$N$ of~$G$ satisfying $K < N < H$. For a chief factor $X=H/K$ and a normal subgroup $N$ of $G$, we say $X$ is \textit{contained} in $N$ if $K< H\leq N$. We will also say $X=H/K$ is \textit{contained} in $G/N$ if $N\leq K<H\leq G$. We will say that $H/K$~is a \textit{complemented} chief factor if there exists a subgroup~$C$ such that $G = HC$ and $H \cap C = K$; that is, when $C/K$~is a complement to the subgroup~$H/K$ in the quotient group~$G/K$. Observe that if the chief factor $H/K$ is abelian, then such a subgroup $C$ must be maximal in $G$. In particular, no abelian chief factor of $G$ contained in the Frattini subgroup $\Frat{G}$ is complemented in $G$.

Given a group $G$, we say that an additively written abelian group~$M$ on which $G$~acts via automorphisms is a \emph{$G$\nbd module}. A $G$\nbd module $M$ is called \textit{trivial} if $m^g=m$ for all $g\in G$ and all $m\in M$. Otherwise, the $G$\nbd module $M$ is called \textit{non-trivial}. We can consider a $G$\nbd module $M$ as a module for the group ring~$\Zint G$. If such a $G$\nbd module $M$~is isomorphic, as a $G$\nbd module, to an abelian chief factor of $G$, then there exists a unique prime $p$ such that $M$ is an elementary abelian $p$\nbd group, and the induced action of $G$ makes $M$ into an irreducible $\Field{p}G$\nbd module. We define $\delta_G(M)$ to be the number of abelian chief factors of $G$ that are isomorphic to $M$ as $G$\nbd modules and complemented in $G$. Using \cite[Theorem~2.1]{Lafuente}, it follows that $\delta_G(M)$ does not depend on choices of chief series of $G$ and hence is an invariant of $G$. 

For a group $G$,  let $I_{G}$ be the kernel of the natural map $\Zint G \to \Zint$ induced by mapping every element of~$G$ to~$1$, which is also called the \emph{augmentation ideal} of the group ring~$\Zint G$. We denote by $d(I_G)$ the minimum number of generators of $I_G$ as a $G$\nbd module.

Following Lucchini's earlier work~\cite{Lu97} on the generation of wreath products, in order to obtain the number of generators of a group~$G$, we shall use the \emph{presentation rank}~$\pr{G}$. It is a non-negative integer by definition in~\cite{Roggenkamp}, but this also follows from Lemma~\ref{lem:pr} below. Our work relies on the following results involving $\pr{G}$. Lemma~\ref{lem:pr} below characterises $\pr{G}$ and links $d(G)$ to $d(I_G)$, the minimum number of generators of the augmentation ideal of $G$ as a $G$\nbd module.

\begin{lem}[Roggenkamp~{\cite[Theorem 2.1]{Roggenkamp}}]\label{lem:pr}
Let $G$ be a finite group. Then
\[
\pr{G}=d(G)-d(I_G).
\]
\end{lem}

\begin{lem}[Gruenberg~{\cite[(B.ii)]{Gru76}}]
  \label{lem:Gruenberg}
  Let $G$~be a finite group. If $\pr{G} > 0$ and $N$~is a soluble normal subgroup of~$G$, then $d(G) = d(G/N)$.
 \end{lem}

For a finite group $G$ and an additively written $G$\nbd module $A$, the first cohomology group~$\Cohom{G,A}$ is equal to the quotient of the group~$\Der{G,A}$ of derivations $G \to A$ by the group~$\Inn{G,A}$ of inner derivations (see, for example, \cite[Chapter~11]{Robinson}). A derivation $\delta \colon G\to A$ is a function from the group $G$ to the additive module $A$ such that for $g, h\in G$, $(gh)^{\delta} = (g^{\delta})^{h} + h^{\delta}$, while an inner derivation $\delta$ is a derivation for which there exists an element $a\in A$ such that for every $g\in G$, $g^{\delta}=a-a^g$. The results developed by Lucchini in \cite{Lu97} rely on $\Cohom{G,M}$, where $M$ is a $G$\nbd module.

We write $\End{M}$ for the endomorphism ring of $M$ as an $\Field{p}G$\nbd module. For an irreducible $\Field{p}G$\nbd module $M$, Schur's Lemma tells us that $\End{M}$ is a division ring containing multiplication scalers. Since $\End{M}$ is finite, it is a field extension of $\Field{p}$. The action of $\End{M}$ on $M$ induces an action of $\End{M}$ on the group of derivations $\Der{G, M}$. Therefore, we can view $M$ and $\Cohom{G,M}$ as vector spaces over $\End{M}$. In this context, Lucchini introduces the following parameters:
\begin{align}
  r_{G}(M) &= \dim_{\End{M}}M \notag \\
  s_{G}(M) &= \dim_{\End{M}}\Cohom{G,M} \notag \\
  h_{G}(M) &= \left\lfloor \frac{s_{G}(M)-1}{r_{G}(M)} \right\rfloor +
  2 \label{eq:h_G}.
\end{align}
The first two parameters are the exponents of the orders of $M$ and $\Cohom{G,M}$ when expressed as powers of the order of $\End{M}$. When we use these parameters, we will show that for the $G$\nbd modules that we are interested in, $\End{M}$ contains only the scalers. Thus $r_G(M)$ is the dimension of $M$. 

Given a group $G$ and a $G$\nbd module $M$, the \textit{centraliser} of $M$ in $G$, denoted by $\Cent{G}{M}$, is the set $\set{g\in G}{m^g=m~\text{for all $m\in M$}}$, and it is a normal subgroup of $G$. We collect useful results from \cite{Lu97} and \cite{Stammbach} in the following. 

\begin{lem}\label{lem:LuStam}
The following statements hold.
\begin{enumerate}
\item\label{i:Lucchini-s}{\normalfont (Lucchini~{\cite[1.2]{Lu97}})}
 Let $M$~be an irreducible $G$\nbd module for a finite group~$G$.
    Then
    \[
    s_{G}(M) = \delta_{G}(M) + \dim_{\End{M}} \Cohom{G/\Cent{G}{M},M}.
    \]
\item\label{i:H1<rGM}{\normalfont (Lucchini~{\cite[1.3]{Lu97}})}
Let $M$ be an irreducible $G$\nbd module. Then
\[
\dim_{\End{M}}H^1(G/C_G(M),M)<r_G(M).
\]
\item\label{i:Stam}{\normalfont (Stammbach~{\cite[Theorem A]{Stammbach}})}
Given a prime number $p$, a finite group $G$ is $p$-soluble if and only if 
\[
\Cohom{G/C_G(M),M}=0
\]
for all irreducible $\Field{p}G$\nbd modules $M$.
\item\label{i:augmentationgenerator}{\normalfont (Lucchini~{\cite[Proposition 1]{Lu97}})}
Let $H$ be a finite group and $G$ be a transitive permutation group of degree $n$. Then
\begin{align*}
d(I_{H\wr G})=\max\left(d(I_{(H/H')\wr G}), \left\lfloor\frac{d(I_H)-2}{n}\right\rfloor+2\right).
\end{align*}
\item\label{i:soluble}{\normalfont (Lucchini~{\cite[Theorem 2]{Lu97}})}
If $H$ is a finite soluble group and $G$ is a transitive permutation group of degree $n$, then
\[
d(H\wr G)=\max\left(d\left(\frac{H}{H'}\wr G\right),\left\lfloor\frac{d(H)-2}{n}\right\rfloor+2\right).
\]
\item\label{i:Lucchini-AwrG}{\normalfont (Lucchini~{\cite[Theorem 4]{Lu97}})}
Let $A\wr G$ be the wreath product of a finite abelian group $A$ by a finite group $G$ constructed via the regular action of $G$. Set 
\[
\rho_p=\max_Mh_G(M)+d_p(A)
\]
where $M$ ranges over the set of non-trivial irreducible $\Field{p}G$\nbd modules, with $\rho_p=0$ if every irreducible $\Field{p}G$\nbd modules is trivial.  Then
\[
d(A\wr G)=\max_{p\mid\order{A}}(d(A\times G),\rho_p).
\]
\end{enumerate}
\end{lem}

In this paper, we are interested in the case that $G$ is an almost simple group with socle $S$. Then in the third part of the lemma above, recall that $G$ is $p$\nbd soluble if and only if $p\nmid\order{S}$.

\section{Wreath products of an abelian group by an almost simple group of Lie type}\label{sec:AwrG}

Lucchini~\cite[Theorem 4]{Lu97} (see Lemma~\ref{lem:LuStam}\ref{i:Lucchini-AwrG}) gives a formula is given for $d(A\wr G)$, where $A$ is a finite abelian group and the wreath product is constructed via the regular action of $G$. We will use this result to investigate the case when $G$ is an almost simple group of Lie type in this section. The formula relies on the maximum of $h_G(M)$ among non-trivial $\Field{p}G$\nbd modules for a fixed prime $p$. By Equation~\eqref{eq:h_G} and Lemma~\ref{lem:LuStam}\ref{i:Lucchini-s}, we can use $\delta_G(M)$ to determine this maximum. We will devote the first half of this section to $\delta_G(M)$. Establishing this, we will then give the formula for $d(A\wr G)$.

Throughout this section, let $S$ be a simple group of Lie type and let $G$ be an almost simple group with socle $S$. We denote by $\Out(S)$ the outer automorphism group of $S$, by $\Outdiag(S)$ the outer diagonal automorphism group of $S$, by $\Phi_S$ the field automorphism group of $S$, and by $\Gamma_S$ the graph automorphism group of $S$.

Consider the normal series $\trivsubgp<S\leq G$, and refine it to the following chief series:
\[
\trivsubgp=Y_0<Y_1=S<Y_2<\dots<Y_n=G.
\]
The chief factors of $G$ isomorphic to $G$\nbd modules necessarily occur as chief factors of $G/S$. As we are only interested in abelian chief factors of $G$, in the following lemma, in order to establish $\delta_G(M)$, we will focus on chief factors of $G/S$. Note that the group $G/S$ is a subgroup of $\Out(S)$, and \cite[Theorem 2.5.12(b)]{GorLyonsSol} shows that
\begin{equation}\label{eq:OutS}
\Out(S)=\Outdiag(S)\rtimes (\Phi_S\Gamma_S),
\end{equation}

\begin{lem}\label{lem:deltaGM}
Let $S$ be a simple group of Lie type, and let $G$ be an almost simple group with socle $S$. Suppose that $M$ is a non-trivial irreducible $\Field{p}G$\nbd module for some prime $p$. Then $\delta_G(M)\leq1$. Furthermore, in each of the following three cases, $\delta_G(M)=0$:
\begin{enumerate}
\item$S=D_4(q)$ for some prime power $q$ and $p\neq 2,3$;
\item $S=D_{2m}(q)$, for some $m\geq3$ and some prime power $q$;
\item $S\neq D_{2m}(q)$ for any $m, q$ and $p$ does not divide $\order{\Outdiag(S)\cap (G/S)}$.
\end{enumerate}
\end{lem}
\begin{prf}
In order to prove $\delta_G(M)\leq1$, it suffices to show that non-trivial abelian chief factors of $G/S$ that are complemented in $G$ are pairwise not $G$\nbd isomorphic. In the following, we will split into three cases according to $S$: $S=D_4(q)$ for some prime power $q$, $S=D_{2m}(q)$ for some $m\geq3$ and prime power $q$ and $S\neq D_{2m}(q)$ for any $m$ and $q$.

Suppose that $S=D_4(q)$. Then, according to \cite[Chapter 2.6]{KleidmanLiebeck}, $\Out(S)=E\times F$, where $F$ is cyclic and $E$ is isomorphic to a subgroup of $S_4$. Suppose that $G/S$ is nilpotent, so $[G/S, G/S]\leq\Frat{G/S}$. Consider a chief series of $G/S$ derived from the normal series $\trivsubgp\leq [G/S, G/S]<G/S$. Then any chief factor between $[G/S, G/S]$ and $G/S$ is also a chief factor of the abelianisation of $G/S$. Hence it is a trivial chief factor, and all non-trivial chief factors of $G/S$ must be contained in $[G/S, G/S]$. It follows that they are inside $\Frat{G/S}$. However, abelian chief factors contained in the Frattini subgroup are not complemented, and therefore, $\delta_G(M)=0$. Now assume that $G/S$ is not nilpotent, whence $E$ is $S_3$, $A_4$ or $S_4$. Then $[G/S, G/S]$ is a subgroup of $C_3$, $C_2^2$ or $A_4$. A chief factor which is non-trivial as a module must have order greater than $2$, and hence non-trivial chief factors in $[G/S, G/S]$ have pairwise different orders. The conclusion follows. Moreover, if $p\neq2, 3$, then $M$ is not $G$\nbd isomorphic to non-trivial chief factors in $[G/S, G/S]$.

Now assume that $S=D_{2m}(q)$ for some $m\geq3$ and some prime power $q$. In this case, using \cite[Chapter 2.6]{KleidmanLiebeck} again, $\Out(S)=E\times F$, where $F$ is cyclic and $E$ is either abelian or isomorphic to $D_8$. Thus, $G/S$ is nilpotent. By the argument above, $\delta_G(M)=0$.

Finally, assume that $S\neq D_{2m}(q)$ for any $m$ and $q$. Theorem 2.5.12 in~\cite{GorLyonsSol} tells us that $\Phi_S\Gamma_S$ in Equation~\eqref{eq:OutS} is abelian. Let $N=(G/S)\cap\Outdiag(S)$. As $(G/S)/N$ is isomorphic to a subgroup of $\Phi_S\Gamma_S$, non-trivial abelian chief factors of $G/S$ must be contained in $N$. Furthermore, as $N$ is a normal subgroup of $G/S$, then $\Frat{N}\leq\Frat{G/S}$ and $\Frat{N}$ is normal in $G/S$. Consider a chief series of $G/S$ derived from the normal series $\trivsubgp\leq \Frat{N}\leq N\leq G/S$. Abelian chief factors of $G/S$ contained in $\Frat{N}$ are not complemented. It follows that non-trivial complemented abelian chief factors of $G$ must appear between $\Frat{N}$ and $N$. Let $H_X=X_2/X_1$ and $H_Y=Y_2/Y_1$ be two abelian complemented chief factors of $G/S$ such that $\Frat{N}\leq X_1< X_2\leq Y_1<Y_2\leq N$. Note that $N$ is cyclic according to \cite[Table 5]{Atlas}, and so $\order{N/\Frat{N}}$ is square-free. Then $\order{H_X}\neq\order{H_Y}$ , whence they are not $G$\nbd isomorphic. Moreover, if $p$ does not divide $\order{N}$, then $M$ is not isomorphic as a $G/S$\nbd module to any abelian chief factor contained in $N$.
\end{prf}

Now that we have determined $\delta_G(M)$, we will establish $\max_M h_G(M)$, where $M$ ranges over the set of non-trivial irreducible $\Field{p}G$\nbd modules. We will use Lemma~\ref{lem:LuStam}\ref{i:Stam} to obtain $\max_M h_G(M)$ based on whether $G$ is $p$\nbd soluble. 

\begin{lem}\label{lem:maxhG}
Let $S$ be a finite simple group of Lie type, and let $G$ be an almost simple group with socle $S$. Let $p$ be a prime. Then
\begin{align*}
\max_M h_G(M)=
\begin{cases}
1 & \text{if $p\nmid\order{S}$},\\
2 & \text{if $p\mid\order{S}$},
\end{cases}
\end{align*}
where $M$ ranges over the non-trivial irreducible $\Field{p}G$\nbd modules.
\end{lem}
\begin{prf}
Assume first that $p\nmid\order{S}$. Since $G$ is $p$\nbd soluble, Lemma~\ref{lem:LuStam}\ref{i:Stam} tells us that $\Cohom{G/C_G(M),M}=0$ for all irreducible $\Field{p}G$\nbd modules $M$. Hence we deduce from Lemma~\ref{lem:LuStam}\ref{i:Lucchini-s} that $s_G(M)=\delta_G(M)$. Substituting this into Equation~\eqref{eq:h_G} gives
\[
h_{G}(M) = \left\lfloor \frac{\delta_{G}(M)-1}{r_{G}(M)} \right\rfloor +  2. 
\]
As $p\nmid\order{S}$, certainly $p\neq 2,3$ when $S=D_4(q)$ according to the order of $D_4(q)$ given in~\cite[Table 6]{Atlas}. On the other hand, when $S\neq D_{2m}(q)$, it follows from \cite[Theorem 2.5.12(c)]{GorLyonsSol} that $p\nmid\order{\Outdiag(S)}$ when $p\nmid\order{S}$. Lemma~\ref{lem:deltaGM} then yields $\delta_G(M)=0$ for all non-trivial irreducible $\Field{p}G$\nbd modules when $p\nmid\order{S}$. Then we get $\max_M h_G(M)=1$.

Now assume that $p\mid\order{S}$. Lemma~\ref{lem:deltaGM} gives $\delta_G(M)\leq1$ for all non-trivial irreducible $\Field{p}G$\nbd modules. We obtain from Lemma~\ref{lem:LuStam}\ref{i:Lucchini-s} and \ref{lem:LuStam}\ref{i:H1<rGM} that 
\begin{align*}
s_G(M)&=\delta_{G}(M) + \dim_{\End{M}} \Cohom{G/\Cent{G}{M},M}\\
&<1+r_G(M).
\end{align*}
Substituting this into Equation~\eqref{eq:h_G} gives
\[
h_{G}(M) = \left\lfloor \frac{s_{G}(M)-1}{r_{G}(M)} \right\rfloor +2<3.
\]
On the other hand, as $p\mid\order{S}$, $G$ is not $p$\nbd soluble and there is a non-trivial $\Field{p}G$\nbd module $M_0$ such that $\Cohom{G/\Cent{G}{M_0},M_0}$ is not trivial by Lemma~\ref{lem:LuStam}\ref{i:Stam}. Thus we deduce that 
\[
\dim_{\End{M_0}} \Cohom{G/\Cent{G}{M_0},M_0}\geq1
\]
and hence $s_G(M_0)\geq\delta_G(M_0)+1$ by Lemma~\ref{lem:LuStam}\ref{i:Lucchini-s}. Using this inequality in Equation~\eqref{eq:h_G} for $M_0$, we obtain
\[
h_G(M_0)\geq\left\lfloor \frac{\delta_G(M_0)}{r_{G}(M_0)} \right\rfloor +2\geq2.
\]
We then conclude that $\max h_G(M)=2$, where $M$ ranges over the non-trivial irreducible $\Field{p}G$\nbd modules.
\end{prf}

Now let $S$ be a non-abelian simple group of Lie type, let $G$ be any almost simple group with socle $S$ and let $A$ be a finite abelian group. Let $W=A\wr G$ be the wreath product of the abelian group $A$ by the almost simple group $G$ constructed via the regular action. We will use Lemma~\ref{lem:LuStam}\ref{i:Lucchini-AwrG} to establish $d(W)$.

\begin{lem}\label{lem:AwrGLietype}
Let $S$ be a simple group of Lie type and let $G$ be an almost simple group with socle $S$. Let $A$ be a finite abelian group and set $W=A\wr G$ to be the wreath product of the abelian group by the almost simple group with respect to the regular action. Then
\[
d(W)=\max_q(d(A\times G), d(A)+1, d_q(A)+2),
\]
where $q$ ranges over the set of the prime numbers dividing $\order{A}$ and $\order{S}$.
\end{lem}
\begin{prf}
By Lemma~\ref{lem:LuStam}\ref{i:Lucchini-AwrG}, 
\[
d(A\wr G)=\max\left( d(A\times G), \max_{p\mid\order{A}}\rho_p\right),
\]
where $\rho_p=d_p(A)+\max_Mh_G(M)$, and $M$ ranges over the non-trivial irreducible $\Field{p}G$\nbd modules. By Lemma~\ref{lem:maxhG}, for every prime $p\mid\order{A}$, 
\begin{align*}
\rho_p=
\begin{cases}
d_p(A)+1 & \text{if $p\nmid\order{S}$},\\
d_p(A)+2 & \text{if $p\mid\order{S}$}.
\end{cases}
\end{align*}
Since $A$ is abelian, $d(A)=\max_{p\mid\order{A}}d_p(A)$. Hence, if $p\mid\order{A}$ and $p\nmid\order{S}$, then $\rho_p=d_p(A)+1\leq d(A)+1$, whereas if $p\mid\order{A}$ and $p\mid\order{S}$, then $\rho_p=d_p(A)+2$. Moreover, choosing a prime $r\mid\order{A}$ such that $d_r(A)=d(A)$, shows that the term $d(A)+1$ is already included in $\max_{p\mid\order{A}}\rho_p$: if $r\nmid\order{S}$, then $\rho_p=d(A)+1$, while if $r\mid\order{S}$, then $\rho_r=d(A)+2$. Therefore
\[
\max_{p\mid\order{A}}\rho_p=\max_q(d(A)+1, d_q(A)+2),
\]
where $q$ ranges over the prime numbers dividing both $\order{A}$ and $\order{S}$. Substituting this into Lemma~\ref{lem:LuStam}\ref{i:Lucchini-AwrG} completes the proof.
\end{prf}

\section{Proof of main theorem}\label{sec:Wsimpleregular}

We will now establish the main theorem. Accordingly, suppose that $k\geq2$, and let $G_1, G_2,\dots, G_k$ be almost simple groups. Define $W=G_k\wr G_{k-1}\wr\dots\wr G_1$ to be the wreath product constructed via the regular action of each factor. We also define $H=G_k\wr G_{k-1}\wr \dots\wr G_2$ and set $A=H/H'$ to be the abelianisation of $H$; that is,
\[
A=(G_k/G_k')\times (G_{k-1}/G_{k-1}')\times\dots\times (G_{2}/G_{2}').
\]

Combining Lemma~\ref{lem:AwrGLietype} with results in~\cite{Lu97} establishes $d(A\wr G_1)$.

\begin{lem}\label{lem:keylemmaregular-abelian}
Let $S$ be the socle of $G_1$. The minimum number of generators of the wreath product $A\wr G_1$ is given by
\[
d(A\wr G_1)=\max_q(d(A\times G_1), d(A)+1, d_q(A)+2),
\]
where $q$ ranges over the common prime divisors of $\order{A}$ and $\order{S}$.
\end{lem}
\begin{prf}
The case when $S$ is a simple group of Lie type is given in Lemma~\ref{lem:AwrGLietype}. If $S$ is an alternating group, then $\Out(S)$ is abelian. According to~\cite[Table 1]{Atlas}, $\Out(S)$ is also abelian when $S$ is one of the sporadic groups. Since $G_1/S$ is a subgroup of $\Out(S)$, and in the alternating and sporadic cases considered here, $\Out(S)$ is abelian, the action of $G_1$ on any abelian chief factor above $S$ is trivial. Therefore such chief factors are trivial as $G_1$\nbd modules. Then for any non-trivial $\Field{p}G_1$\nbd module $M$, $\delta_{G_1}(M)=0$. In this case, the formula is a direct consequence of~\cite[Corollary 5]{Lu97}.
\end{prf}

The key step of the main theorem is the following lemma, which reduces $d(W)$ to $d(A\wr G_1)$.

\begin{lem}\label{lem:regularinduction}
  Let $W$~and~$A$ be as defined above. Then
  \begin{equation}
    d(W) = d(A \wr G_{1}).
    \label{eq:regularreduction}
  \end{equation}
\end{lem}
\begin{prf}
We first observe that $A\wr G_1$ is a homomorphic image of $W=H\wr G_1$, and so $d(W) \geq d(A \wr G_{1})$. We will prove the reverse inequality in the rest of the proof.

Set $L=G_{k-1}\wr G_{k-2}\wr \dots\wr G_1$ acting on $\Gamma=G_{k-1}\times \dots\times G_1$ such that $W=G_k\wr_{\Gamma} L$. Denote the socle of $G_k$ by $F$. Then $F$ is the unique minimal normal subgroup of $G_k$, it is non-abelian simple, and the group $L$ acts transitively on $\Gamma$; hence $N=F^{\Gamma}$ is the unique minimal normal subgroup of $G_k\wr_{\Gamma} L$. Theorem~\ref{thm:LuMe} then yields
\[
d(W)=d(W/N)=d((G_k/ F)\wr L).
\]
It suffices to prove that $d((G_k/ F)\wr L)\leq d(A \wr G_{1})$, and we will use induction in the following.

Suppose that $k=2$, and then $W/N=(G_2/F)\wr G_1$. Note that 
\[
(G_2/F)/(G_2/F)'\cong G_2/(G_2'F).
\]
As $F$ is non-abelian simple, $F=F'\leq G_2'$ and $(G_2/F)/(G_2/F)'\cong G_2/G_2'\cong A$. Using the fact that $G_2/F$ is soluble, Lemma~\ref{lem:LuStam}\ref{i:soluble} yields 
\begin{align*}
d((G_2/F)\wr G_1)&=\max\left(d(((G_2/F)/(G_2/F)')\wr G_1),\left\lfloor\frac{d(G_2/F)-2}{\order{G_1}}\right\rfloor+2\right)\\
&=\max\left(d(A\wr G_1),\left\lfloor\frac{d(G_2/F)-2}{\order{G_1}}\right\rfloor+2\right).
\end{align*}
As $2\leq d(G_1), d(G_2)\leq 3$ by \cite[Corollary of Theorem 1, p.195]{DaLu}, 
\[
\left\lfloor\frac{d(G_2/F)-2}{\order{G_1}}\right\rfloor+2\leq \left\lfloor\frac{1}{\order{G_1}}\right\rfloor+2\leq d(G_1)\leq d(A\wr G_1).
\]
We therefore establish that $d((G_2/F)\wr G_1)=d(A\wr G_1)$.

Now assume that $k>2$ and that the lemma holds for any iterated wreath product of almost simple groups with fewer than $k$ factors. We shall prove the lemma by splitting into two cases according to the presentation rank of $\bar W=(G_k/ F)\wr L$.

Suppose first that $\pr{\bar W}=0$. Set $\bar H=(G_k/F)\wr G_{k-1}\wr\dots\wr G_2$ such that $\bar W=\bar H\wr G_1$. Note that $(G_k/F)/(G_k/F)'\cong G_k/G_k'$ because $F\leq G_k'$. It follows that $\bar H/\bar H'\cong A$. Then Lemma~\ref{lem:pr} gives $d(\bar W)=d(I_{\bar W})$. Substituting this into Lemma~\ref{lem:LuStam}\ref{i:augmentationgenerator} yields
\begin{align}
d(\bar W)&=\max\left(d(I_{A\wr G_1}), \left\lfloor\frac{d(I_{\bar H})-2}{\order{G_1}}\right\rfloor+2\right)\notag\\
&\leq\max\left(d(A\wr G_1), \left\lfloor\frac{d(\bar H)-2}{\order{G_1}}\right\rfloor+2\right). \label{eq:dbarW=max}
\end{align}
Set $B=G_k/G_k'\times\dots\times G_3/G_3'$ so that $d(H)=d(B\wr G_2)$ by the induction hypothesis. Using the fact that $d(G_2)\leq 3$ from \cite[Corollary of Theorem 1, p.195]{DaLu}, we deduce that
\[
d(\bar H)\leq d(H)\leq d(B)+d(G_2)\leq d(A)+3.
\]
Lemma~\ref{lem:keylemmaregular-abelian} gives that $d(A\wr G_1)\geq d(A)+1$, and so
\[
\left\lfloor\frac{d(\bar H)-2}{\order{G_1}}\right\rfloor+2\leq \left\lfloor\frac{d(A\wr G_1)}{\order{G_1}}\right\rfloor+2\leq d(A\wr G_1).
\]
Then $d(\bar W)\leq d(A\wr G_1)$ by Equation~\eqref{eq:dbarW=max}, and hence we have established Equation~\eqref{eq:regularreduction}.

Suppose now that $\pr{\bar W}>0$. As $\bar W=(G_k/F)\wr L$ and $G_k/F$ is soluble, it follows from Lemma~\ref{lem:Gruenberg} that $d(\bar W)=d(L)$. Set 
\[
C=(G_{k-1}/G_{k-1}')\times \dots\times (G_2/G_{2}').
\] Then according to the induction hypothesis, $d(L)\leq d(C\wr G_1)\leq d(A\wr G_1)$. We conclude that $d(\bar W)\leq d(A\wr G_1)$ regardless of $\pr{\bar W}$ and the proof is completed.
\end{prf}

Combining Lemmas~\ref{lem:keylemmaregular-abelian} and \ref{lem:regularinduction} gives the main theorem of this paper.

{\small


\begin{thebibliography}{99} \setlength{\itemsep}{-0.2ex}

\bibitem{Bhatt}
  M. Bhattacharjee. The probability of generating certain
  profinite groups by two elements. \textit{Israel
    J. Math.},~\textbf{86} (1994), no.~1--3, pp. 311–329.

\bibitem{Bond}
  I. V. Bondarenko. Finite generation of iterated wreath
  products. \textit{Arch.\ Math.},~\textbf{95} (2010), pp. 301--308.

\bibitem{Atlas}
J. H. Conway, R. T. Curtis, S. P. Norton, R. A. Parker, R. A. Wilson. \textit{ $\Bbb{ATLAS}$ of finite groups}. Oxford University Press, Eynsham, 1985.

\bibitem{DaLu}
F. Dalla~Volta, A. Lucchini. Generation of almost simple groups. \textit{J. Algebra}, {\bf 178} (1995), no.~1, pp. 194–223. 

\bibitem{GorLyonsSol}
D. Gorenstein, R. Lyons, R. Solomon. \textit{The classification of the finite simple groups}. Mathematical Surveys and Monographs, 40.1. American Mathematical Society, Providence, RI, 1994.

\bibitem{Gru76}
K. W. Gruenberg. Groups of non-zero presentation rank. \textit{Symposia Math.}, {\bf 17} (1976), pp. 215–224.

\bibitem{KleidmanLiebeck}
P. Kleidman, M. Liebeck. \textit{The subgroup structure of the finite classical groups}. London Mathematical Society Lecture Note Series, 129, Cambridge University Press, Cambridge, 1990

\bibitem{Lafuente}
J.~P. Lafuente. Nonabelian crowns and Schunck classes of finite groups. \textit{Arch. Math. (Basel)}, {\bf 42} (1984), no.~1, pp. 32–39.

 
\bibitem{LuQuick}  
  J. Lu, M. Quick. Generation of iterated wreath products constructed from alternating, symmetric and cyclic groups. \textit{Internat. J. Algebra Comput}, {\bf 35} (2025), no.~5, pp. 713--731.

\bibitem{Lu97}
A. Lucchini. Generating wreath products and their augmentation ideals. \textit{Rend. Sem. Mat. Univ. Padova}, {\bf 98} (1997), pp. 67–87.

\bibitem{LuMe}
A. Lucchini, F. Menegazzo. Generators for finite groups with a unique minimal normal subgroup. \textit{Rend. Sem. Mat. Univ. Padova}, {\bf98} (1997), pp. 173–191.

\bibitem{MRQ}
  M. Quick. Probabilistic generation of wreath products of
  non-abelian finite simple groups,~II.
  \textit{Internat.\ J. Algebra Comput.}, \textbf{16} (2006), no.~3,
  pp. 493--503.

\bibitem{Robinson}
D.~J.~S. Robinson. {\it A course in the theory of groups}. Second edition,
Graduate Texts in Mathematics, 80, Springer, New York, 1996.

\bibitem{Roggenkamp}
K. W. Roggenkamp. Relation modules of finite groups and related topics. \textit{Algebra i Logika}, \textbf{12} (1973), pp. 351--369.

\bibitem{Stammbach}
U. Stammbach. Cohomological characterisations of finite solvable and nilpotent groups. \textit{J. Pure Appl. Algebra}, \textbf{11} (1977), pp. 293--301.

\bibitem{Wiegold}
J. Wiegold. Growth sequences of finite groups. \textit{J. Austral. Math. Soc.}, {\bf 17} (1974), pp. 133--141.

\bibitem{Woryna}
A. Woryna. The rank and generating set for inverse limits of wreath products of Abelian groups. \textit{Arch. Math. (Basel)}, {\bf 99} (2012), no.~6, pp. 557--565.

\end{thebibliography}
\end{document}